\documentclass{article}
\usepackage{amsthm}
\usepackage{amssymb}%
\usepackage{amsfonts}%
\usepackage{amsmath}%
\usepackage{graphicx}
\usepackage[utf8]{inputenc}
\usepackage[T1]{fontenc}
\usepackage[hidelinks]{hyperref}
\graphicspath{ {./images/} }

\newtheorem{theorem}{Theorem}[section]

\newtheorem{definition}{Definition}
\newtheorem{example}{Example}

\newtheorem{lemma}{Lemma}

\newtheorem{proposition}{Proposition}
\newtheorem{remark}{Remark}

\numberwithin{equation}{section}

\begin{document}
\title{Surfaces and Hypersurfaces with Prescribed Radial Mean Curvature}

\author{\large \bf
Armando M. V. Corro\\
IME, Universidade Federal de Goi\'as \\
Caixa Postal 131, 74001-970, Goi\^ania, GO, Brazil\\
e-mail:corro@ufg.br
  \and
 \large \bf Marcelo Lopes Ferro    \\
 IME, Universidade Federal de Goi\'as \\
Caixa Postal 131, 74001-970, Goi\^ania, GO, Brazil\\
e-mail: marceloferro@ufg.br
 }
\date{}
\maketitle \thispagestyle{empty}

\begin{abstract}
In this work, we provide a local classification of certain special classes of surfaces determined by the prescription of the radial mean curvature in terms of the height and angle functions. Moreover, we introduce a special class of hypersurfaces, and we also provide a local classification of these three-dimensional hypersurfaces whose second mean curvature vanishes. Finally, we present a recursive method for constructing such hypersurfaces, extending the same curvature prescription approach to higher dimensions.
\end{abstract}
\maketitle
\textbf{keywords:}\,\,{Differential geometry, Surfaces, Hypersurfaces, radial mean curvature}

\section{Introduction}
\label{sec::Intro}
The study of the properties of curves in $\mathbb{R}^2$ can be regarded as a particular case of more general concepts in differential geometry, such as the prescription of the mean curvature in surfaces of revolution or in surfaces with a canonical principal direction (see \cite{Carretero2024,MunteanuNistor2011,Garnica2012,Lopez2019,CorroFerroPreprint2025}).

Motivated by the classification of surfaces with a canonical principal direction given in \cite{MunteanuNistor2011,Garnica2012}, we extend the notion of CPD surfaces to higher dimensions by introducing a new class of hypersurfaces in $\mathbb{R}^{n+1}$, called \emph{rotated translational hypersurfaces}. The hypersurfaces in this class are composed of two hypersurfaces: one is the directrix in $\mathbb{R}^{s+1}$, with $s < n$, and the other is the profile in $\mathbb{R}^{\,n-s+1}$.

The main result of this work is Theorem \ref{teorema1}, which establishes a general relation characterizing the class of rotated translational hypersurfaces of arbitrary dimension $n$, with $H_{n-1}=0$. Its three-dimensional version, Theorem \ref{teorema3}, provides a complete local classification of all such hypersurfaces with vanishing second mean curvature, representing the main geometric contribution of the paper. In addition, Theorem \ref{teorema2} introduces a recursive construction method that allows the generation of higher-dimensional examples from lower-dimensional ones.

As preliminary results, we also provide local classifications for \emph{CPD surfaces}, \emph{translational surfaces}, and \emph{harmonic surfaces of graphic type}, under the same curvature prescription approach. For further details on harmonic surfaces of graphic type, see \cite{RiverosCorroBarbosa2016,RiverosCorro2018}.

The paper is organized as follows. In Sect. \ref{sec:preliminares}, we define the rotated translational hypersurfaces and present the two main theorems. 
In Sect. \ref{sec:edo}, we establish two technical lemmas on ordinary differential equations which will be used in the following sections. 
In Sect. \ref{sec:superficie}, we provide local classifications of CPD, translational, and harmonic surfaces of graphic type under an appropriate curvature prescription. 
Sect. \ref{sec:paralela} contains important technical results on parallel hypersurfaces, which will be essential for the proofs in Sect. \ref{sec:RTH}.
Finally, in Sect. \ref{sec:RTH}, we prove the two main results and, as a consequence, give a local classification of all three-dimensional rotated translational hypersurfaces whose second mean curvature vanishes. 
The paper concludes with a recursive method for constructing higher-dimensional examples.

\section{Preliminaries and statement of the main results}\label{sec:preliminares}

In this section, we begin by presenting the concept of radial mean curvature, then provide the definition of a rotated translational hypersurface, and subsequently highlight the height function and the angle function, which will be used for the prescription of the radial mean curvature. Finally, we present the two main results of this work.\\

Let $M$ be a hypersurface of dimension $n$, with principal curvatures $\kappa_i$. 
The principal radial curvatures are given by $\displaystyle{\tfrac{1}{\kappa_i}}$, the mean radial curvature $H_R$ is the average of the principal radial curvatures, and the sum of the principal radial curvatures is denoted by $A_M$. From the previous definitions, we restrict ourselves to an open set, where $H_n\neq 0$, we have

\begin{eqnarray}
  A_M\,=\,n  H_R \,=\,n\frac{H_{n-1}}{H_n} =\sum_{i=1}^n \frac{1}{\kappa_i},\nonumber
\end{eqnarray}
where $H_r$ are the $r$-mean curvature of $M$, given by
\[
H_r = \frac{S_r(W)}{\binom{n}{r}},
\]
where, for integers $0 \leq r \leq n$, $S_r(W)$ is given by
\[
S_0(W) = 1, \qquad
S_r(W) = \sum_{1 \leq i_1 < \cdots < i_r \leq n} \kappa_{i_1}\cdots \kappa_{i_r}.
\].

\begin{definition}
 A hypersurface $M \subset \mathbb{R}^{n+1}$ is called a 
\textbf{rotated translational hypersurface} if there exist:

\begin{itemize}
    \item $P_{s+1}$, a fixed vector subspace of dimension $s+1$ in $\mathbb{R}^{n+1}$,  $1 \leq s \leq n-1$,
    \item $V \in P_{s+1}$, a constant unit vector, such that, $\displaystyle{P_{t+1} = P_{s+1}^{\perp} \oplus \text{span}\{V\}}$, a vector subspace of 
    dimension $t+1$ with $s+t=n$,
\end{itemize}
and two hypersurfaces:
\begin{itemize}
    \item $M_1$, a hypersurface in $P_{s+1}$, with Gauss map $N_1$,
    \item $M_2$, a hypersurface in $P_{t+1}$.
\end{itemize}
Such that the hypersurface $M$ is obtained, for each $p_1 \in M_1$, by applying
an orthogonal transformation that leaves invariant $P_{s+1}^{\perp}$, determined by 
$N_1(p_1)$, which maps $M_2$ into a hypersurface in
\[
P_{s+1}^{\perp} \oplus \text{span}\{N_1(p_1)\},
\]
followed by a translation also determined by $p_1$.   
\end{definition}

Note that, as defined above, the hypersurface $M$ is foliated by hypersurfaces contained in hyperplanes of dimension $t+1$, all of them isometric to $M_2$.

The hypersurfaces $M_1$ and $M_2$ are called, respectively, the directrix of $M$ and the profile of $M$.

\begin{remark}
   The Rotated translational hypersurfaces have the following properties:
\begin{enumerate}
    \item When $s = 1$ and $t = 1$, we obtain the surfaces with a canonical principal direction ( CPD surfaces ). For further details, see \cite{MunteanuNistor2011,Garnica2012,Lopez2019} .
    \item When $s = n-1$ and $t = 1$, we obtain a class of hypersurfaces with a canonical principal direction, studied by Palmas, see \cite{Garnica2012}.
    \item When $M_1 = S^{s}(r)$, the hypersurface $M$ is invariant under the subgroup of orthogonal transformations that leaves fixed $P_{s+1}^{\perp}$.
\end{enumerate}
\end{remark}

\begin{remark}
    Up to a rigid motion of $\mathbb{R}^{n+1}$ we may consider: $\{e_1,\ldots,e_{n+1}\}$ the canonical basis of $\mathbb{R}^{n+1}$, 
$P_{s+1} = \text{span}\{e_1,\ldots,e_{s+1}\} = \mathbb{R}^{s+1}$, 
$V = e_1$, 
$P_{t+1} = \text{span}\{e_1, e_{s+2}, \ldots, e_{s+t+1}\}$, 
which we can identify with $\mathbb{R}^{t+1}$. 

Then, the directrix hypersurface, the profile hypersurface, and the rotated translational 
hypersurface $M$ can be locally parametrized by
\begin{eqnarray}
    &&Y(u) = (Y_1(u), \ldots, Y_{s+1}(u)), 
\quad u = (u_1,\ldots,u_s) \in U \subset \mathbb{R}^s,\label{eqY}\\
&&Z(v) = (Z_1(v), \ldots, Z_{t+1}(v)), 
\quad v = (v_1,\ldots,v_t) \in V \subset \mathbb{R}^t,\label{eqZ}\\
&&X(u,v) = Y(u) + Z_1(v)N^Y(u) 
        + \sum_{r=2}^{t+1} Z_r(v) e_{s+r}, 
        \qquad s+t=n,\label{eqX}
\end{eqnarray}
where $N^Y$ denotes the Gauss normal map of $Y$, and $X$ is defined on the points of $U \times V$ such that $1 - Z_1 \kappa_i \neq 0$ for all $1 \leq i \leq s$, being $\kappa_i$ the principal curvatures of the directrix $M^s_1$.

\end{remark}

One of the purposes of this work is to classify \emph{Rotated Translational Hypersurfaces} with the prescription of $H_R$ in terms of the 
\begin{enumerate}
    \item []\hspace{2cm} height function $X_1=\big<X,e_1\big>$,
    \item[]\hspace{2cm} angle function $N_1=\big<N,e_1\big>$.
\end{enumerate}
In terms of these functions, the prescription of $H_R$, in an open set where $H_n\neq 0$, is given by
\[N_1A_X=aX_1+b,\quad\text{where}\quad a,b\in\mathbb{R},\quad A_X=nH_R.\]

\begin{remark}
    If $a=b=0$, then the hypersurface has zero $(n-1)$-mean curvature, that is, $H_{n-1}=0$. 
In the case with $n=2$, this prescription yields CPD minimal surfaces. The surface obtained in this case is the catenoid, as expected, since Munteanu and Nistor proved in \cite{MunteanuNistor2011} that the only minimal surface of CPD in $\mathbb{R}^3$, besides the plane, is the catenoid.
.
\end{remark}

In what follows, whenever the prescription of the radial mean curvature $H_R$ is mentioned, 
it is understood that this takes place on an open subset of $M$ where $H_n \neq 0$.\\

The following theorems will be proved in Section \ref{sec:RTH}.

\vspace{0.2cm}

\begin{theorem}\label{teorema1}
Let $M^n$ be a rotated translational hypersurface of dimension $n \geq 3$, whose profile is a curve. Then, the $(n-1)$-mean curvature of $M^n$ is zero, that is, $H_{n-1}=0$, if and only if, there exist constants $C_1 > 0$, $C_2 \in \mathbb{R}$, such that the directrix of $M^n$ is either a hypersurface $M_1^{\,n-1}$ with vanishing $(n-2)$-mean curvature, or a $(n-1)$-sphere of radius $r$, $S^{n-1}(r)$, and the profile of $M$ can be locally parameterized by
\begin{eqnarray}\label{Zteorema1}
    Z(\theta)=\begin{cases}
   \left(C_1\sec^{n-1}\theta\,,\,C_2+(n-1)C_1\int\sec^{n-1}\theta d\theta\right) , &\text{if}\quad M_1^{n-1}\neq \mathbb{S}^{n-1}(r)  ,\\[1.2em]
    \left(-r+C_1\sec^{n-1}\theta\,,\,C_2+(n-1)C_1\int\sec^{n-1}\theta d\theta\right) , &\text{if}\quad M_1^{n-1}= \mathbb{S}^{n-1}(r)
\end{cases}
\end{eqnarray} 
where $\theta \in \left(-\tfrac{\pi}{2},\,\tfrac{\pi}{2}\right)$.

\end{theorem}

\vspace{0.2cm}

\begin{theorem}\label{teorema2}
Let $M^n$ be a rotated translational hypersurface parametrized by (\ref{eqX}) 
on an open subset where $H_n \neq 0$, whose directrix and profile hypersurfaces 
are parametrized by (\ref{eqY}) and (\ref{eqZ}), respectively. Then
\[N^{X}_{1}A_{X} \;=\; aX_{1}+b\]
if and only if there exists a constant $C \in \mathbb{R}$ such that 
 \begin{eqnarray}
     N^{\widetilde{Y}}_{1}A_{\widetilde{Y}} &=& a\widetilde{Y}_{1}+b, \quad\text{where}\quad \widetilde{Y}=Y-\frac{C}{s+a}N^Y,\nonumber\\
N^{Z}_{1}A_Z &=& (a+s)Z_{1}+C,\nonumber
 \end{eqnarray}
where $N^X$, $N^Y$ and $N^Z$ denote the Gauss maps of $X$, $Y$ and $Z$, respectively, 
and $X_1 = \langle X, e_1\rangle$, $Y_1 = \langle Y, e_1\rangle$, $Z_1 = \langle Z, e_1\rangle$, 
with $a,b \in \mathbb{R}$ constants.
\end{theorem}

\vspace{0.2cm}

\begin{remark}
    As an application of Theorem \ref{teorema1}, we will classify all 3-dimensional rotated translational hypersurfaces with $H_2=0$. Moreover, as an application of Theorem \ref{teorema2}, we will set up a recursive process to construct rotated translational hypersurfaces $M^n$ with prescribed $A_M$ in terms of the height function and the angle function.
\end{remark}

\section{Auxiliary results on ordinary differential equations}\label{sec:edo}

In this section we present two lemmas that provides explicit solutions for a special type of ordinary differential equations. This results will be of fundamental importance in the following sections.

Precisely, we present a result that provides us with more suitable coordinates to classify, in the following sections, certain surfaces and hypersurfaces with the prescription of the radial mean curvature in terms of the height and angle functions.

\vspace{0.3cm}

\begin{lemma}\label{lemaedo}
    Let $\widetilde{a}$, $\widetilde{b}$, $\widetilde{C}$ be constants and let $f:I\subseteq\mathbb{R}\to\mathbb{R}$ be a $C^2$ function satisfying
\begin{eqnarray}
    \sqrt{1+f'^{2}}\left(\frac{1+f'^{2}}{f''} - \widetilde{a} f - \widetilde{b}\right)=\widetilde{C}.\label{eqhpreliminar}
\end{eqnarray}
Then, there exists a change of coordinates $u=u(\theta)$ and constants $\widetilde{C}_1$, $\widetilde{C}_2\in\mathbb{R}$, such that $h$ is given by
\begin{equation}\label{solEDO}
\hspace{-0.3cm}
\begin{cases}
\begin{alignedat}{2}
& f(u)\,=\,\widetilde{C}_1 \sec^{\widetilde{a}}\theta \;-\; \tfrac{\widetilde{C}}{\widetilde{a}+1}\cos\theta \;-\; \tfrac{\widetilde{b}}{\widetilde{a}}, \\[0.3em]
& u(\theta)\,=\,\widetilde{C}_2 \;+\; \tfrac{\widetilde{C}}{\widetilde{a}+1}\sin\theta \;+\; \widetilde{a}\,\widetilde{C}_1 \!\int \sec^{\widetilde{a}}\theta\,d\theta
\end{alignedat}
& \widetilde{a} \in \mathbb{R}\setminus\{-1,0\}, \\[1.2em]
\begin{alignedat}{2}
& f(u)\,=\,\widetilde{C}_1 \;-\; \widetilde{C}\cos\theta \;-\; \widetilde{b}\,\ln(\cos\theta) \\[0.3em]
&  u(\theta)\,=\,\widetilde{C}_2 \;+\; \widetilde{b}\,\theta \;+\; \widetilde{C}\,\sin\theta
\end{alignedat}
& \widetilde{a}=0, \\[1.2em]
\begin{alignedat}{2}
& f(u)\,=\,\widetilde{b} \;+\; \widetilde{C}_1\cos\theta \;-\; \widetilde{C}\,\cos\theta\,\ln(\cos\theta), \\[0.3em]
& u(\theta)\,=\,\widetilde{C}_2 - \widetilde{C}_1\sin\theta + \widetilde{C}\!\left[\sin\theta\,\ln(\cos\theta)+\ln\!\big(\sec\theta+\tan\theta\big)\right]
\end{alignedat}
& \widetilde{a}=-1.
\end{cases}
\end{equation}
with $\displaystyle{\theta\in J\subseteq\left(-\frac{\pi}{2}\,,\,\frac{\pi}{2}\right)}$, where
\begin{eqnarray}\label{J}
   J=\begin{cases}
    \bigg\{\theta\in\bigg(\frac{-\pi}{2}\,,\,\frac{\pi}{2}\bigg)\,\bigg|\,\widetilde{a}\widetilde{C}_1\sec^{\widetilde{a}+1}\theta+\frac{\widetilde{C}}{\widetilde{a}+1}\neq0 \bigg\}, &\text{if}\quad \widetilde{a} \in \mathbb{R}\setminus\{-1,0\} ,\\[1.2em]
    \bigg\{\theta\in\bigg(\frac{-\pi}{2}\,,\,\frac{\pi}{2}\bigg)\,\big|\,\widetilde{b}\sec\theta+\widetilde{C}\neq0\bigg\}, &\text{if}\quad \widetilde{a}=0 ,\\[1.2em]
\bigg\{\theta\in\bigg(\frac{-\pi}{2}\,,\,\frac{\pi}{2}\bigg)\,\big|\,\widetilde{C}\big(1+\ln\cos\theta\big)-\widetilde{C}_1\neq0\bigg\}, &\text{if}\quad \widetilde{a}=-1,
\end{cases}
\end{eqnarray} 
\end{lemma}

\begin{proof}
    Consider the equation
    \begin{eqnarray}
    \sqrt{1+f'^{2}}\left(\frac{1+f'^{2}}{f''} - \widetilde{a} f - \widetilde{b}\right)=\widetilde{C}.\nonumber
\end{eqnarray}
Let $f'=\tan\theta$, with $\theta=\theta(u)\in J\subseteq\bigg(\tfrac{-\pi}{2},\,\tfrac{\pi}{2}\bigg)$, where $\displaystyle{J=\bigg\{\theta\in\bigg(\tfrac{-\pi}{2},\,\tfrac{\pi}{2}\bigg)\,\bigg|\,\theta'\neq0\bigg\}}.$\\ 
Thus, $f''=\theta'\sec^2\theta$ and the equation above can be written as
\begin{eqnarray}
    \frac{1}{\theta'}-\widetilde{a}f-\widetilde{b}=\widetilde{C}\cos\theta,
    \quad\text{i.e.}\quad 
    \theta'=\frac{1}{\widetilde{C}+\widetilde{a}f+\widetilde{b}}.\label{thetalinha}
\end{eqnarray}
Since $\displaystyle{f'=\frac{dg}{d\theta}\,\theta'}$, then
\[
\frac{df}{d\theta}=\frac{f'}{\theta'}=\frac{\tan\theta}{\theta'}
=\tan\theta\big(\widetilde{C}+\widetilde{a}f+\widetilde{b}\big).
\]

Therefore, we obtain the linear ordinary differential equation
\begin{eqnarray}
    \frac{df}{d\theta}-\widetilde{a}\tan\theta\,f
    =\widetilde{b}\tan\theta+\widetilde{C}.\label{eqlinear}
\end{eqnarray}

\begin{enumerate}
    \item[i)] If $\widetilde{a}\neq -1$ and $\widetilde{a}\neq 0$, the solution of (\ref{eqlinear}) is given by
    \[
    f=\frac{-\widetilde{b}}{\widetilde{a}}
    -\frac{\widetilde{C}}{\widetilde{a}+1}\cos\theta
    +\widetilde{C}_1\sec^{\widetilde{a}}\theta.
    \]
    From (\ref{thetalinha}), we have
    \[
    \frac{du}{d\theta}
    =\frac{\widetilde{C}}{\widetilde{a}+1}\cos\theta
    +\widetilde{a}\widetilde{C}_1\sec^{\widetilde{a}}\theta,
    \quad\text{i.e.}\quad 
    u=\widetilde{C}_2+\frac{\widetilde{C}}{\widetilde{a}+1}\sin\theta
    +\widetilde{a}\widetilde{C}_1\int\sec^{\widetilde{a}}\theta\,d\theta.
    \]
    Moreover, by (\ref{thetalinha}), $\displaystyle{J=\bigg\{\theta\in\bigg(\tfrac{-\pi}{2},\,\tfrac{\pi}{2}\bigg)\,\bigg|\,
    \widetilde{a}\widetilde{C}_1\sec^{\widetilde{a}+1}\theta
    +\frac{\widetilde{C}}{\widetilde{a}+1}\neq0 \bigg\}}$.

    \item[ii)] If $\widetilde{a}=0$, the solution of (\ref{eqlinear}) is given by
    \[
    f=\widetilde{C}_1-\widetilde{C}\cos\theta-\widetilde{b}\ln\cos\theta.
    \]
    From (\ref{thetalinha}), we obtain
    \[
    u=\widetilde{C}_2+\widetilde{b}+\widetilde{C}\sin\theta,
    \quad\text{and}\quad 
    J=\bigg\{\theta\in\bigg(\tfrac{-\pi}{2},\,\tfrac{\pi}{2}\bigg)\,\bigg|\,
    \widetilde{b}\sec\theta+\widetilde{C}\neq 0 \bigg\}.
    \]

    \item[iii)] If $\widetilde{a}=-1$, then the solution of (\ref{eqlinear}) is given by
    \[
    f(u)=\widetilde{b}+\widetilde{C}_1\cos\theta-\widetilde{C}\cos\theta\,\ln(\cos\theta).
    \]
    Therefore, from (\ref{thetalinha}) we get
    \[
    u=\widetilde{C}_2-\widetilde{C}_1\sin\theta
    +\widetilde{C}\!\left[\sin\theta\,\ln(\cos\theta)+\ln\!\big(\sec\theta+\tan\theta\big)\right],
    \]
    and
    \[
    J=\bigg\{\theta\in\bigg(\tfrac{-\pi}{2},\,\tfrac{\pi}{2}\bigg)\,\big|\,
    \widetilde{C}(1+\ln\cos\theta)-\widetilde{C}_1\neq0\bigg\}.
    \]
\end{enumerate}

\end{proof}

\begin{remark}\label{edo2}
   By Lemma (\ref{lemaedo}), if $\widetilde{C}=0$, then the solution of the equation
\[
h''(a_1h+b_1)=1+(h')^2,
\]
is given by
\begin{eqnarray}\label{h}
h=\begin{cases}
\frac{-b_1}{a_1}+C_1\,\sec^{\,a_1}\theta,\quad\text{and}\quad u=C_2+a_1\,C_1\!\int\sec^{\,a_1}\theta\,d\theta, &\text{if}\quad a_1\neq0,\\[1.2em]
-b_1\,\ln(\cos\theta)+C_2,\quad\text{and}\quad u=b_1\,\theta+C_1, &\text{if}\quad a_1=0,\,\,\,b_1\neq0,
\end{cases}
\end{eqnarray}
In this case, $\displaystyle{\theta\in J=\bigg(\tfrac{-\pi}{2},\,\tfrac{\pi}{2}\bigg)}$.

\end{remark}

Note that, when $a_1=1$,  we obtain a family of translated catenaries
\[x\,=\,C_1\cosh\bigg(\frac{y-C_2}{C_1}\bigg)-b_1,\quad C_1>0,\,\,\, C_2\,,\,b_1\in\mathbb{R}.\]
In fact, by the previous Remark, writing
\[
x=-C+C_1\,\sec\theta \quad \text{and} \quad y=C_2+\,C_1\ln\big(\sec\theta+\tan\theta\big),
\]
we obtain that
\[
\cosh\!\left(\frac{y-C_2}{C_1}\right)=\sec\theta.
\]
Moreover, when $a_1=2$, we obtain families of parabolas 
\[x\,=\,\frac{\big(y-C_2\big)^2}{4C_1}+\bigg(C_1-\frac{b_1}{2}\bigg)^2,\quad C_1\neq0,\,\,\, C_2\,,\,b_1\in\mathbb{R}.\]

\begin{lemma}\label{edocomplexa}
Let $a,b,c\in\mathbb{R}$ be constants and let $f,g:U\subset\mathbb{C}\to\mathbb{C}$ be holomorphic on a domain $U$ with $g'(z)\neq0$.
Assume that
\begin{equation}\label{eq:complex-ode}
 (f')^{2}=(a f+b+ic)\Big(f''-\frac{g''}{g'}\,f'\Big)\qquad\text{on }U,
\end{equation}
where the primes denote derivatives with respect to $z$.
Then, on each simply connected subdomain where a branch of the logarithm and of the complex power is fixed, there exist constants $z_1,z_2\in\mathbb{C}$ such that
\begin{eqnarray}\label{g(z)f(z)}
   f(z)=\begin{cases}
    \frac{\big(z_1\,g(z)+z_2\big)^{\frac{a}{a-1}}-b-ic}{a}, &\text{if}\quad a \in \mathbb{R}\setminus\{1,0\}\\[1.2em]
    z_1\,e^{z_2\,g(z)}\,-\,b-ic, &\text{if}\quad a=1 ,\\[1.2em]
-\,(b+ic)\,\ln\!\big(z_1\,g(z)+z_2\big), &\text{if}\quad a=0.
\end{cases}
\end{eqnarray} 
\end{lemma}

\begin{proof}
   Consider $f,g:U\subset\mathbb{C}\to\mathbb{C}$ holomorphic on a domain $U$ with $g'(z)\neq0$ and the equation
    \[
\ (f')^2 = (af+b+ic)\Bigg(f''-\frac{g''}{g'}f'\Bigg)\ .
\]

We write $f=f(g)$, so that
\[
f' = \frac{df}{dg}\,g', 
\quad
f'' = \frac{d^{2}f}{dg^{2}}g'^{2} + \frac{df}{dg}\,g''.
\]
Substituting into the previous equation, we have
\[
\left(\frac{df}{dg}\right)^{2} = (a f + b + i c)\frac{d^{2}f}{dg^{2}}.
\]
Let $q=\dfrac{df}{dg}$. Then the last equation can be rewritten as
\begin{eqnarray}
    q^{2}dg = (a f + b + i c)\frac{dq}{dg}dg
\quad\text{i.e.}\quad
q\,df = (a f + b + i c)dq.\label{eqq2}
\end{eqnarray}

\begin{enumerate}
    \item [(i)] If $a\neq 0$ and $a\neq 1$, then from (\ref{eqq2}), we obtain
\[
\frac{df}{dg} = \widetilde{z}_1(a f + b+ic)^{1/a},\quad \widetilde{z}_1\in\mathbb{C}.
\]
Thus, we get the first equation of the (\ref{g(z)f(z)}).

\item[(ii)] If $a= 1$, then from (\ref{eqq2}), we obtain
\[\frac{df}{dg}=\widetilde{z}_1(f+b+ic).\]
Then, by integrating this last equation, we obtain the second equation in (\ref{g(z)f(z)}).

\item[(iii)] If $a= 0$, then using (\ref{eqq2}) and proceeding as in the previous cases, we conclude that $f$ is given by the third equation in (\ref{g(z)f(z)}), which completes the proof.

\end{enumerate}

\end{proof}

\section{CPD, Translational, and Harmonic Surfaces with radial mean curvature prescribed in terms of the height and angle functions}\label{sec:superficie}

In this section, we deal with surfaces with canonical principal direction
(CPD surfaces), translational surfaces, and Laguerre-type surfaces, 
which are obtained from the prescription of the radial mean curvature given in terms
of the height and angle functions. These surfaces will be useful in the next section
for providing examples of rotated translational hypersurfaces. We begin by presenting
the definition of CPD surfaces.

\begin{definition}
Let $d \in \mathbb{R}^3$ be a unitary vector. We say that a surface
$M^2$ in $\mathbb{R}^3$ has a canonical principal direction with respect to $d$ 
if the tangent part $d^T$ along $M$ is a principal direction. 
\end{definition}
Surfaces with a canonical principal direction have suitable parameterization. 
It was proved in \cite{Garnica2012} that a CPD surface with respect to $d$ admits the following parameterization
\[
X(u,v) = \gamma(u) + f(v)n(u) + g(v)d,
\]
where $\gamma = \gamma(u)$, $s \in I \subset \mathbb{R}$ is a curve contained in a plane orthogonal to $d$, 
$n = n(u)$ is a unit vector field along $\gamma$ and orthogonal to both $\gamma$ and $d$. The curve $\gamma$ is named the directrix of $M$, and the curve $\beta(v)=(f(v)\,,\,g(v))$ is the profile
curve of $M$.\\

Locally, given a CPD surface, we may consider it with $d = e_3$ and directrix
and profile curves of the graph type. Hence, we consider the profile curve
$(h(v)\,,\, v)$ and the directrix curve $\gamma(u)=(f(u)\,,\,u)$. Under these conditions, the CPD surface is locally parameterized by

\begin{eqnarray}
    X(u,v) = (f(u)\,,\,u\,,\,0) +  h(v)n(u) +ve_3. \label{CPD}
\end{eqnarray}

The unit normal vector of this CPD surface is given by
\[N=\frac{1}{\sqrt{1+h'^2}}\big(h'e_3-n\big),\]
 where $n(u)$ and the curvature of $\gamma$ are given by
 \begin{eqnarray}
     n=\frac{\big(-1\,,\,f'\,,\,0\big)}{\sqrt{1+f'^2}},\quad \kappa_{\gamma}=\frac{-f''}{(1+f'^2)^{3/2}},\label{kgamma}
 \end{eqnarray}
and the principal curvatures of the $X$, are given by
\begin{eqnarray}
    \kappa_1=\frac{k_{\gamma}}{(1-k_{\gamma}h)\sqrt{1+h'^2}},\quad\text{and}\quad\kappa_2=\frac{h''}{(1+h'^2)^{3/2}}.\label{principais}
\end{eqnarray}
Therefore, the height and angle functions of $X$ are given by
\begin{eqnarray}
    X_1=f(u)-\frac{h(v)}{\sqrt{1+f'(u)^2}},\quad N_1=\frac{1}{\sqrt{1+f'(u)^2}\sqrt{1+h'(v)^2}}.\label{cpdfun}
\end{eqnarray}

Next, we provide the theorem that gives a local classification of CPD surfaces
with the prescription of the radial mean curvature in terms of the height and
angle functions.

\begin{theorem}\label{TeoremaCPD}
   Let $M^2$ be a CPD surface locally parameterized by (\ref{CPD}), where the Gaussian curvature is nonzero. Then, the radial mean curvature of $M^2$ is given by
\[
N_1A_M=aX_1+b,\quad\text{where}\quad a,b\in\mathbb{R},
\]
if and only if there exists a constant $C$ and local coordinates $(\theta,\phi)\in U$, such that the directrix and profile curves of $M^2$ are given by
\begin{equation}\label{gammacpd}
\gamma(\theta)=
\begin{cases}
\left(
\begin{alignedat}{2}
& C_1 \sec^{a}\theta \;-\; \tfrac{C}{a+1}\cos\theta \;-\; \tfrac{b}{a}, \\[0.3em]
& C_2 \;+\; \tfrac{C}{a+1}\sin\theta \;+\; a\,C_1 \!\int \sec^{a}\theta\,d\theta
\end{alignedat}
\right),
& a \in \mathbb{R}\setminus\{-1,0\}, \\[1.2em]
\left(
\begin{alignedat}{2}
& C_1 \;-\; C\cos\theta \;-\; b\,\ln(\cos\theta), \\[0.3em]
& C_2 \;+\; b\,\theta \;+\; C\,\sin\theta
\end{alignedat}
\right),
& a=0, \\[1.2em]
\left(
\begin{alignedat}{2}
& b \;+\; C_1\cos\theta \;-\; C\,\cos\theta\,\ln(\cos\theta), \\[0.3em]
& C_2 - C_1\sin\theta + C\!\left[\sin\theta\,\ln(\cos\theta)+\ln\!\big(\sec\theta+\tan\theta\big)\right]
\end{alignedat}
\right),
& a=-1.
\end{cases}
\end{equation}

and
\begin{eqnarray}\label{betacpd}
    \beta(\phi)=\begin{cases}
    \left(C_3\,\sec^{\,1+a}\phi-\frac{C}{a+1}\,,\,C_4+(1+a)\,C_3\!\int\sec^{\,1+a}\phi\,d\phi\right), &\text{if}\quad a\neq-1 ,\\[1.2em]
\bigg(-C\,\ln(\cos\phi)+C_3\,,\,C_4+C\,\phi\bigg), &\text{if}\quad a=-1,
\end{cases}
\end{eqnarray}
where $\displaystyle{U=\bigg\{(\theta,\phi)\in J\times\Big(\tfrac{-\pi}{2},\,\tfrac{\pi}{2}\Big)\,\Big|\,1-\kappa_{\gamma}\beta_1(\phi)\neq0\bigg\}}$, with $J$ given by (\ref{J}), and with constants $C_1\neq0$, $C_3\neq0$, and $C_2,C_4\in\mathbb{R}$.
\end{theorem}

\begin{proof}
    Let $M^2$ be a CPD surface locally parameterized by (\ref{CPD}). Thus, by (\ref{principais}) and (\ref{kgamma}), we get
    \[A_M=\frac{1}{\kappa_1}+\frac{1}{\kappa_2}=-\sqrt{1+h'^2}\left(\frac{-(1+f'^2)^{3/2}}{f''}-h+\frac{1+h'^2}{h''}\right).\]
    In this way,  by (\ref{cpdfun}), we have that, $N_1 A_R^2 = a X_1 + b$,
if and only if
\begin{eqnarray}
    (1+a)\,h
+ \sqrt{1+f'^{2}}\left(\frac{1+f'^{2}}{f''} - a f - b\right)
- \frac{1+h'^{2}}{h''}
= 0.\label{eqCPD1}
\end{eqnarray}
Hence, there exists a constant $C$ such that (\ref{eqCPD1}) is equivalent to
\begin{eqnarray}
    &&\sqrt{1+f'^{2}}\left(\frac{1+f'^{2}}{f''} - a f - b\right)=C,\label{eqcpdedo}\\
    &&
\frac{1+h'^{2}}{h''}=(1+a)h\,
+ C.\label{eqcpd2}
\end{eqnarray}
Therefore, by Lemma \ref{lemaedo} and Remark \ref{edo2}, we conclude the proof.

\end{proof}

Munteanu and Nistor proved in \cite{MunteanuNistor2011} that the only minimal surface in $\mathbb{R}^3$ with a canonical principal direction besides the plane is the catenoid. 

Note that the catenoid arises as a particular case of Theorem \ref{TeoremaCPD} by choosing $a=b=0$.

 In fact, if $a=b=0$, then by Theorem \ref{TeoremaCPD} the directrix curve is a circle, thus the CPD surface $M^2$ is a surface of revolution. Moreover, the profile curve is the catenary, since
\[\beta(\phi)=
    \left(C_3\,\sec\phi-C\,,\,C_4+C_3\ln(\sec\phi+\tan\phi)\right).\]
Thus
\begin{eqnarray}
    y=C_4+C_3\ln\bigg[\frac{x+C}{C_3}+\sqrt{\bigg(\frac{x+C}{C_3}\bigg)^2-1}\bigg]=C_4+C_3\operatorname{arcosh}\bigg(\frac{x+C}{C_3}\bigg).\nonumber
\end{eqnarray}
That is, $x = y(x)$ is the catenary, and $M^2$ is the catenoid.

\vspace{0.3cm}

Next, we highlight translational surfaces and then provide a local classification 
for translational surfaces with the prescription of the radial mean curvature in 
terms of the height and angle functions.\\

Let $M^2$ be a translational surface. Then, locally, $M^2$ can be parametrized by
\begin{eqnarray}
    X(u,v)=(f_1(u)+f_2(v)\,,\,u\,,\,v).\label{translcional}
\end{eqnarray}
The unit normal field and the Gaussian and mean curvatures of $M^2$ are given by
\begin{eqnarray}
    N=\frac{(1\,,\,-f_1'\,,\,-f_2')}{\sqrt{1+f_1'^2+f_2'^2}},\quad K=\frac{f_1''f_2''}{(1+f_1'^2+f_2'^2)^2},\quad H
=\frac{(1+f_2'^2)f_1''+(1+f_1'^2)f_2''}{2(1+f_1'^2+f_2'^2)^{3/2}}.\label{NkHtran}
\end{eqnarray}

For this surfaces, the height and angle functions of $X$ are given by
\begin{eqnarray}
    X_1=f_1(u)+f_2(v),\quad N_1=\frac{1}{\sqrt{1+f_1'(u)^2+f_2'(v)^2}}.\label{tranfun}
\end{eqnarray}

\begin{theorem}\label{Teorematran}
    Let $M^2$ be a translational surface locally parameterized by (\ref{translcional}), where the Gaussian curvature is nonzero. Then, the radial mean
curvature of $M^2$ is given by
\[N_1A_M=aX_1+b,\quad\text{where}\quad a,b\in\mathbb{R},\]
if and only if there exists local coordinates $\displaystyle{(\theta,\phi)\in\bigg(\frac{-\pi}{2}\,,\,\frac{\pi}{2}\bigg)\times\bigg(\frac{-\pi}{2}\,,\,\frac{\pi}{2}\bigg) }$, 
such that $M^2$ is given by
\begin{eqnarray}\label{Xtranslacional}
    X=\begin{cases}
    \bigg(\frac{-b}{a}+C_3\sec^{a}\theta
+C_5\sec^{a}\phi\,,\,C_4+aC_3\Phi(\theta)\,,\,
C_6+aC_5\Phi(\phi)\bigg), &\text{if}\quad a\neq 0 ,\\[1.2em]
\bigg(C_3-C_1\ln\big(\cos\theta\big)-(b-C_1)\ln\big(\cos\phi\big)\,,\,C_4+C_1\theta\,,\, \,C_5+(b-C_1)\phi\bigg), &\text{if}\quad a=0,
\end{cases}
\end{eqnarray}
where $\displaystyle{\Phi(t)=\int\sec^at\,dt}$.
\end{theorem}

\begin{proof}
    Let $M^2$ be a translational surface locally parameterized by (\ref{translcional}). Thus, by (\ref{NkHtran}), we get
    \[A_R^2=\frac{1}{\kappa_1}+\frac{1}{\kappa_2}=\frac{2H}{K}=\sqrt{1+f_1'^2+f_2'^2}\left(\frac{1+f_2'^2}{f_2''}+\frac{1+f_1'^2}{f_1''}\right).\]
    Thus,  by (\ref{tranfun}), we have $N_1 A_R^2 = a X_1 + b$,
if and only if
\begin{eqnarray}
   \frac{1+f_2'^2}{f_2''}+\frac{1+f_1'^2}{f_1''}
= a(f_1+f_2)+b.\label{eqtran1}
\end{eqnarray}
Hence, there exists constants $C_1$ and $C_2$, with $C_1+C_2=b$  such that (\ref{eqtran1}) is equivalent to
\begin{eqnarray}\nonumber
   1+f_1'^2=(af_1+C_1)f_1'',\qquad
1+f_2'^2=(af_2+C_2)f_2''.
\end{eqnarray}
Therefore, by Lemma \ref{lemaedo}, we conclude that the translational surface $M^2$ can be parameterized by (\ref{Xtranslacional}).
\end{proof}

\begin{remark}
    Based on the choices of the constants $a$ and $b$, from Theorem \ref{Teorematran} we have:
\begin{enumerate}
    \item[a)] If $a = b = 0$, then the translational surface is minimal. In this case, according to Theorem \ref{Teorematran}, $M^2$ is the Scherk surface. In fact, by Theorem \ref{Teorematran}, if $a = b = 0$, then the translational surface is parameterized by
\[X(\theta,\phi)=\bigg(C_3-C_1\ln\big(\cos\theta\big)+C_1\ln\big(\cos\phi\big)\,,\,C_4+C_1\theta\,,\, \,C_5-C_1\phi\bigg), \]
\[X(\theta,\phi)=\bigg(C_3-C_1\ln\frac{\cos\theta}{\cos\phi}\,,\,C_4+C_1\theta\,,\, \,C_5-C_1\phi\bigg). \]
\item [b)] If $a=2$, then by Theorem \ref{Teorematran}, the translational surface is an elliptic paraboloid if 
$C_3 C_5 > 0$, and a hyperbolic paraboloid if $C_3 C_5 < 0$.\\
In fact, if $a = 2$, then by Theorem \ref{Teorematran}, we have
\[
X(\theta,\phi) \;=\;
\left(
\frac{-b}{2} \;+\; \frac{C_3}{2}\sec^{2}\theta \;+\; \frac{C_5}{2}\sec^{2}\phi,\;
C_4 + C_3\tan\theta,\;
C_6 + C_5\tan\phi
\right).
\]
Thus
\[
X(\theta,\phi) \;=\;\left( \frac{-b}{2}\,,\,C_4\,,\,C_6\right)+
\left(
\frac{C_3}{2}\sec^{2}\theta \;+\; \frac{C_5}{2}\sec^{2}\phi,\;
 C_3\tan\theta,\;
 C_5\tan\phi
\right).
\]
Clearly, an elliptic paraboloid if $C_3 C_5 > 0$, and a hyperbolic paraboloid if $C_3 C_5 < 0$.
\end{enumerate}
\end{remark}

\vspace{0.3cm}

We conclude this section by providing a local classification for harmonic surfaces of graph type, with the prescription of the radial mean curvature in terms of the height and angle functions.

\begin{definition}
Let \(X(u,v)=(X_1(u,v),\,X_2(u,v),\,X_3(u,v))\) be a parametrization of a surface \(M^2\subset\mathbb{R}^3\).
We say that \(M^2\) is a \emph{harmonic surface} if \(\Delta X_i=0\) for \(i=1,2,3\).
In other words, each coordinate function of the immersion is harmonic.
\end{definition}

Let $z=u+iv$. Given a holomorphic function $f:\mathbb{C}\to\mathbb{C}$, we adopt the following notation
\[
\langle 1,\,f(z)\rangle=\Re(f(z)),\qquad \langle i,\,f(z)\rangle=\Im(f(z)).
\]

\begin{definition}
Let \(f,g:\mathbb{C}\to\mathbb{C}\) be holomorphic functions.  
The surface \(M\subset\mathbb{R}^3\) parametrized by
\[
X(z)=\big(\langle 1,\,f(z)\rangle\,,\,g(z)\big)
\]
is called a \emph{harmonic surface of graph type}.
\end{definition}
Riveros, Corro, and Barbosa proved in \cite{RiverosCorroBarbosa2016} that
the Gauss map of a graph-type harmonic surface is given by
\begin{eqnarray}
  N = \frac{1}{D}(|g'|^2\,,\,-g'\, \overline{f'}),\quad\text{where}\quad D = |g'|\sqrt{|g'|^2 + |f'|^2}. \label{normalhormonica} 
\end{eqnarray}
Moreover, in \cite{RiverosCorroBarbosa2016}, Riveros, Corro, and Barbosa proved that the Gaussian curvature and the mean curvature of a harmonic surface of graph type are given by
\begin{eqnarray}
    K = -\frac{|g'|^4 |A|^2}{D^4},\quad H = -\frac{|g'|^2 \,\langle A,f'^2\rangle}{2D^3}, \quad\text{where}\quad A = f'' - \frac{g''}{g'}f'.\label{KeHarmonico}
\end{eqnarray}

Next, we provide a local classification of all harmonic surfaces of graph type with prescribed radial mean curvature in terms of the height and angle functions.

\vspace{0.2cm}

\begin{theorem}\label{Teoremaharmonico}
  Let \(f,g:\mathbb{C}\to\mathbb{C}\) be holomorphic functions, such that, $g'(z)\neq0$. 
    Consider $M^2$ a graph-type harmonic surface
where the Gaussian curvature is nonzero. Then, the radial mean
curvature of $M^2$ is given by
\[N_1A_M=aX_1+b,\quad\text{where}\quad a,b\in\mathbb{R},\]
if and only if there exists constants $z_1\neq0,\,z_2\in\mathbb{C}$, 
such that $M^2$ is given by
\[X(z) = (\langle 1, f(z) \rangle\,,\,g(z)),\]
where
\begin{eqnarray}\label{g(z)f(z)}
   f(z)=\begin{cases}
    \frac{\big(z_1\,g(z)+z_2\big)^{\frac{a}{a-1}}-b-ic}{a}, &\text{if}\quad  a \in \mathbb{R}\setminus\{1,0\}\\[1.2em]
    z_1\,e^{z_2\,g(z)}\,-\,b-ic, &\text{if}\quad a=1 ,\\[1.2em]
-\,(b+ic)\,\ln\!\big(z_1\,g(z)+z_2\big), &\text{if}\quad a=0,
\end{cases}
\end{eqnarray} 
with $z_2\neq0$ if $a=1$.
\end{theorem}

\begin{proof}
    Consider the harmonic surface of graph type
\[
X(z) = \big(\langle 1, f(z)\rangle,\, g(z)\big),
\]
with $f$ and $g$ holomorphic functions. Thus, from (\ref{normalhormonica}) and (\ref{KeHarmonico}), we have that the radial mean
curvature of $M^2$ is given by
\[N_1A_M=aX_1+b,\quad\text{where}\quad a,b\in\mathbb{R},\]
if and only if
\[\langle 1,\frac{(f')^2}{A}\rangle=\langle 1,af+b\rangle, \quad\text{where}\quad A = f'' - \frac{g''}{g'}f'.\]
Hence, there exists a real constant $c$ such that
\[
\frac{(f')^2}{A} = af+b+ic,\quad\text{i.e.}\quad f'^2=(af+b+ic)\bigg(f'' - \frac{g''}{g'}f'\bigg).
\]
Therefore, by Lemma \ref{edocomplexa}, we conclude the proof.

\end{proof}

\begin{remark}
    If $a=b=0$, then $\displaystyle{A_M=\frac{1}{\kappa_1}+\frac{1}{\kappa_2}=\frac{2H}{K}}=0.$ That is, $M^2$ is a minimal surface. In this case, as proved in \cite{RiverosCorroBarbosa2016}, $M^{2}$ is the helicoid.
\end{remark}

\section{Technical results on parallel hypersurfaces}\label{sec:paralela}

In this section, we present important technical results on parallel hypersurfaces, which will be used in the subsequent section.  
Given a hypersurface $M^m$ parameterized by $Y:U\subseteq\mathbb{R}^{m}\to M^m$ with Gauss map $N$ and principal curvatures $\kappa_i$, $1\leq i\leq m$, the parallel hypersurface $\widetilde{M}^m$ to $M^m$ is parameterized by $\widetilde{Y}=Y+tN$ on an open subset of $\mathbb{R}^m$ where $1-t\kappa_i\neq0$ for all $1\leq i\leq m$.  

Next, we provide two lemmas.

\begin{lemma}\label{lema1}
    Let $M^m$ be a non–totally umbilical hypersurface such that $\dfrac{H_{m-1}}{H_m}=C$, with $C \in \mathbb{R}$ a constant. 
Then, there exists a parallel hypersurface $\widetilde{M}_C^m$ to $M^m$ such that $\widetilde{H}_{m-1}=0$.
\end{lemma}

\begin{proof}
    Let $\kappa_i$ be the principal curvatures of $M^m$ and consider $Y:U\subset \mathbb{R}^{m}\to M^m$ a parameterization of $M^m$. 
Let $\widetilde{M}_t^m$ be a hypersurface parallel to $M^m$. 
Then, $\widetilde{M}_t^m$ is parameterized by
\[
\widetilde{Y}(u)=Y(u)+tN(u).
\]
Then the principal curvatures of $\widetilde{M}_t^m$ are given by
\begin{eqnarray}
    \widetilde\kappa_i = \frac{\kappa_i}{1-t\kappa_i},\quad\text{equivalently}\quad \kappa_i = \frac{\widetilde\kappa_i}{1+t\widetilde\kappa_i}.\label{principalParalela}
\end{eqnarray}
Since $\displaystyle{s\frac{H_{s-1}}{H_s}=\sum_{i=1}^s\frac{1}{\kappa_i}}$, it follows that $\frac{H_{s-1}}{H_s}=C$ if and only if 
\[
sC=\sum_{i=1}^s\frac{1}{\kappa_i}=\sum_{i=1}^s\left(\frac{1}{\widetilde\kappa_i}\right)+st,\quad\text{i.e.}\quad s\frac{\widetilde{H}_{m-1}}{\widetilde{H}_{m}}=\sum_{i=1}^s\frac{1}{\widetilde\kappa_i}=s(C-t).
\]
Therefore, the hypersurface $\widetilde{M}_C^m$ parameterized by $\widetilde{Y}=Y+CN$ is parallel to $M$ such that $\widetilde{H}_{m-1}=0$.

\end{proof}

\begin{lemma}\label{lema2}
    Let $M^m$ be a hypersurface which is not totally geodesic, with a parameterization $Y:U\subset \mathbb{R}^{m}\to M^m$ and Gauss map $N^Y$. Then, 
  \[N_1^Y\big(A_Y+C\big)=aY_1+b,\quad a\,,\,b\,,C \in \mathbb{R},\]  
      where $a\neq -m$, $N_1^Y=\big<N^Y,e_1\big>$ and $Y_1=\big<Y,e_1\big>$, respectively, the angle and height functions, if and only if there exists a hypersurface $\widetilde{M}^m$ parallel to $M^m$, 
parameterized by 
\[
\widetilde{Y}(u)=Y(u)-\frac{C}{a+m}N^Y,
\] 
such that
\[N_1^{\widetilde{Y}}A_{\widetilde{Y}}=a\widetilde{Y}_1+b,\] 
where $\widetilde{Y}_1$ and $N_1^{\widetilde{Y}}$ are, respectively, the height and angle functions.
\end{lemma}

\begin{proof}
    Let $Y:U\subset \mathbb{R}^{m}\to M^m$ and $N^Y$ the Gauss map. 
Consider 
\[
\widetilde{Y}(u)=Y(u)+tN^Y
\]
the parametrization of the hypersurface $\widetilde{M}^m$ parallel to $M^m$. 
Then, if $\kappa_i$ denotes the principal curvatures of $M^m$, then the principal curvatures of $\widetilde{M}^m$ are given by $\displaystyle{\widetilde{\kappa}_i=\frac{\kappa_i}{1-t\kappa_i}}$.
Thus,
\[
A_Y=\sum_{i=1}^m \frac{1}{\kappa_i}=A_{\widetilde{Y}}+mt.
\]
Hence, we obtain
\[
N_1^Y(A_Y+C)=aY_1+b
\]
if and only if
\[
N_1^Y(A_{\widetilde{Y}}+mt+C)=a{\widetilde{Y}}_1-a\,tN_1^Y+b,
\]
equivalently,
\[
N_1^YA_{\widetilde{Y}}=a{\widetilde{Y}}_1+b
\quad\text{with}\quad t=\frac{-C}{m+a}.
\]

\end{proof}

\begin{remark}
The parallel hypersurfaces $\widetilde{Y}=Y+tN^Y$, are defined on the open set
\[
U_t=\{u\in U\,|\,\ 1-t\,\kappa_i(u)\neq0,\,\,\,\,\,\, \, 1\leq i\leq m\}.
\]
\end{remark}

\section{Rotated translational hypersurfaces}\label{sec:RTH}

In this section, we present a class of hypersurfaces generated by two hypersurfaces, which we call a \textbf{rotated translational hypersurfaces}; a directrix hypersurface and a profile hypersurface. Moreover, we provide a method to construct these hypersurfaces with the prescription of the radial mean curvature in terms of the height and angle functions. Next, we present the proofs of Theorems \ref{teorema1} and \ref{teorema2}, and we conclude this section by classifying the rotated translational hypersurfaces of dimension $3$ and by constructing
some examples of such hypersurfaces with dimensions greater than $3$.

We begin with a proposition that provides the principal curvatures of twisted translational hypersurfaces.

\begin{proposition}\label{Prop1}
    Let $M^n$ be a rotated translational hypersurface parametrized by (\ref{eqX}). The principal curvatures of $M^n$ are given by
\begin{eqnarray}
    \kappa_i = \frac{N^Z_1 \,\kappa^Y_i}{1 - Z_1 \kappa^Y_i} 
\quad \text{and} \quad 
\kappa_A = \kappa^Z_A,\label{RTHprincipal}
\end{eqnarray}
where $\kappa^Y_i$, $1\leq i\leq s$ and $\kappa^Z_A$, $s+1\leq A\leq n$ are, respectively, the principal curvatures of the directrix hypersurface and the profile hypersurface, given by (\ref{eqY}) and (\ref{eqZ}).
\end{proposition}

\begin{proof}
   Consider the parametrization of the rotated translational hypersurface given by
\begin{eqnarray}
   \hspace{-0,5cm} X:U\times V\subseteq\mathbb{R}^{n}\to \mathbb{R}^{n+1},\qquad
    X(u,v)=Y(u)+Z_1(v)N^Y(u)+\sum_{r=2}^{n-s+1}Z_{r}(v)e_{s+r}.\label{Xdef}
\end{eqnarray}
where $Y:U\subseteq\mathbb{R}^{s}\to \mathbb{R}^{s+1}$ and $Z:V\subseteq\mathbb{R}^{\,n-s}\to \mathbb{R}^{\,n-s+1}$ 
are parameterizations of the directrix and profile hypersurfaces, respectively, and 
$N^Y$ denotes the unit normal field of $Y$.

Since the principal curvatures are independent of the parameterization, suppose without loss of generality that
$Y$ is a parameterization by lines of curvature. Hence
\[N^Y_{,i}=-\kappa_i^YY_{,i}.\]
Differentiating (\ref{Xdef}), with respect to $u_i$ and $v_A$, where $u=(u_1,\ldots,u_s)$ and $v=(v_{s+1},v_{s+2},\ldots,v_n)$, we obtain
\[
X_{,i} = (1-Z_1 \kappa_i^Y)Y_{,i}, 
\quad \text{and} \quad 
X_{,A} = Z_{1,A}N^Y + \sum_{r=2}^{n-s+1} Z_{r,A} e_{s+r},
\]
Thus, writing
\[
N^Z = N^Z_1(v)e_{s+1} + \sum_{r=2}^{n-s+1} N_r^Z(v) e_{s+r},
\qquad 
(N^Z_1)^2 + \sum_{r=2}^{n-s+1} (N^Z_r)^2 = 1,
\]
we find that the unit normal field of $X$ is given by
\[
N = N^Z_1(v)N^Y(u) + \sum_{r=2}^{n-s+1} N_r^Z(v) e_{s+r}.
\]
The principal curvatures of $X$ are given by
\[\kappa_j=\frac{-\big<N_{,j}\,,\,X_{,j}\big>}{\big<X_{,j}\,,\,X_{,j}\big>}=\frac{\big<N\,,\,X_{,jj}\big>}{\big<X_{,j}\,,\,X_{,j}\big>},\quad 1\leq j\leq n.\]
Differentiating $N$, with respect to $u_i$ and differentiating $X_{,A}$, with respect to $v_A$, where $u=(u_1,\ldots,u_s)$ and $v=(v_{s+1},v_{s+2},\ldots,v_n)$, we get
\[N_{,i}=N^Z_1(v)N^Y_{,i}=-\kappa_i^YN^Z_1(v)Y_{,i},\quad X_{,AA}=Z_{1,AA}N^Y(u) + \sum_{r=2}^{n-s+1} Z_{r,AA} e_{s+r}.\]
Therefore, for $1 \leq i \leq s$, we have $\kappa_i$ given by (\ref{RTHprincipal}). 
Moreover, we conclude the proof by noting that
\[\kappa_A=\frac{\big<N\,,\,X_{,AA}\big>}{\big<X_{,A}\,,\,X_{,A}\big>}=\frac{\big<N^Z\,,\,Z_{,AA}\big>}{\big<Z_{,A}\,,\,Z_{,A}\big>}=\kappa_A^Z,\quad 1+s\leq A\leq n.\]

\end{proof}

\begin{remark}
    Let $M^n$ be a rotated translational hypersurface. Then, by Proposition \ref{Prop1}, 
we have that the mean curvature of $M^n$ is given by
\begin{eqnarray}
    H = \frac{1}{n}\left( \sum_{i=1}^s \frac{N^Z_1 \kappa^Y_i}{1 - Z_1 \kappa^Y_i} 
      + \sum_{B=s+1}^n \kappa^Z_B \right) 
   = \frac{N^Z_1}{n} \sum_{i=1}^s \frac{\kappa^Y_i}{1 - Z_1 \kappa^Y_i} 
     + \frac{n-s}{n} H^Z.\label{H}
\end{eqnarray}
\end{remark}

The following result provides a relation between the radial mean curvatures of the rotated translational hypersurface and those of the profile and directrix hypersurfaces.

\begin{proposition}\label{Prop2}
Let $M^n$ be a rotated translational hypersurface with directrix hypersurface $M_1^s$ and profile hypersurface $M_2^{\,n-s}$, given by (\ref{eqX}). Then
    \begin{eqnarray}
   A_X= \frac{A_Y}{N_1^Z}-\frac{sZ_1}{N_1^Z}+A_Z,\label{AX}
\end{eqnarray}
\end{proposition}

\begin{proof}
For definition, $\displaystyle{A_X=n\frac{H_{n-1}}{H_n}}$, them
    \begin{eqnarray}
    n\frac{H_{n-1}}{H_n} 
    &=& \sum_{i=1}^n \frac{1}{\kappa_i} 
       = \sum_{i=1}^s \frac{1}{\kappa_i} + \sum_{A=s+1}^n \frac{1}{\kappa_A} 
       = \sum_{i=1}^s \frac{1 - Z_1 \kappa^Y_i}{N^Z_1 \kappa^Y_i} 
         + \sum_{A=s+1}^n \frac{1}{\kappa^Z_A} \nonumber\\
    &=& \frac{1}{N^Z_1}\left(\sum_{i=1}^s \frac{1}{\kappa^Y_i} 
         - \sum_{i=1}^s \frac{Z_1 \kappa^Y_i}{\kappa^Y_i}\right) 
         + (n-s)\frac{H^Z_{n-s-1}}{H^Z_{n-s}} \nonumber\\
    &=& \frac{s}{N^Z_1}\frac{H^Y_{s-1}}{H^Y_s} 
         - \frac{sZ_1}{N^Z_1} 
         + (n-s)\frac{H^Z_{n-s-1}}{H^Z_{n-s}}. \nonumber
\end{eqnarray}
That is, we have (\ref{AX}).
\end{proof}

The following result shows that the minimal hypersurfaces $\Sigma^s\times\mathbb{R}^{n-s}$, where $\Sigma^s$ is a minimal hypersurface in $\mathbb{R}^{s+1}$, are rotated translational hypersurfaces.

\begin{theorem}
Consider the rotated translational hypersurface $M^n$ parameterized by (\ref{eqX}), with the directrix hypersurface $M_1^{s}$ and the profile hypersurface $M_2^{\,n-s}$. Suppose that $M_1^s$ is not totally geodesic. If the profile hypersurface $M_2^{\,n-s}$ is minimal in $\mathbb{R}^{\,n-s+1}$, then $M^n$ is minimal in $\mathbb{R}^{n+1}$ if and only if, up to isometries, $M^n$ is locally a
\[
\Sigma_c^s \times \mathbb{R}^{\,n-s},
\]
where $\Sigma_c^s$ is a minimal hypersurface parallel to $M_1^s$, with a fixed constant $c\in\mathbb{R}$, such that $1-c\kappa_i^Y\neq0$, for $1\leq i\leq s$.

\end{theorem}

%$\Sigma_c^s=\displaystyle{\Big\{ Y(u) + c N^Y(u) \;\big|\; u \in U \Big\}}$

\begin{proof}
    Suppose that the profile hypersurface $M_2^{\,n-s}$ is given by $\displaystyle{Z(v) = \left(Z_1(v), Z_{2}(v), \ldots, Z_{n-s+1}(v)\right)}$
is a minimal hypersurface in $\mathbb{R}^{\,n-s+1}$, hence $H^Z=0$.  
Therefore, by (\ref{H}), $M^n$ is a minimal hypersurface in $\mathbb{R}^{\,n+1}$ if and only if
\[
\sum_{i=1}^s \frac{\kappa^Y_i}{1 - Z_1 \kappa^Y_i} = 0.
\]
Differentiating with respect to $v_A$, we obtain
\[
Z_{1,A}\sum_{i=1}^s \frac{(\kappa^Y_i)^2}{(1 - Z_1 \kappa^Y_i)^2} = 0,\quad s+1\leq A\leq n.
\]
Since $M_1^s$ is not totally geodesic, we conclude that $Z_1(v) = c$ is constant, so that
\[
Z(v) = \left(c, Z_{2}(v), \ldots, Z_{n-s+1}(v)\right)
\]
is a parameterization of $\mathbb{R}^{\,n-s}$.  
Therefore, $M^n$ is a minimal hypersurface in $\mathbb{R}^{\,n+1}$ if and only if
\[
\sum_{i=1}^s \frac{\kappa^Y_i}{1 - c \kappa^Y_i} = 0,
\]
and $Z(v)$ is a parameterization of $\mathbb{R}^{\,n-s}$.  

Thus, we conclude the proof, noting that the principal curvatures of the parallel hypersurface
\[
\Psi(u) = Y(u) + cN^Y(u)
\]
are given by
\[
\lambda_i = \frac{\kappa^Y_i}{1 - c \kappa^Y_i},
\]
where $\kappa^Y_i$ are the principal curvatures of $M_1^s$.

\end{proof}

We now provide the proof of Theorem \ref{teorema1}.

\begin{proof}[Proof of Theorem \ref{teorema1}]
Consider the rotated translational hypersurface $M^n$ locally parameterized by (\ref{eqX}), with the directrix hypersurface $M_1^{\,n-1}$ and the profile being a curve, locally parameterized by $Y$ and $Z(v) = (h(v),v)$, respectively. Thus, we have
\[X(u,v) \,\,=\,\, Y(u) \,\,+\,\, h(v)\, N^Y(u)\,\, +\,\, v\, e_{n+1}.\]

Hence we obtain
\[
N_1^Z = \frac{1}{\sqrt{1+h'^2}}
\quad \text{and} \quad 
\kappa_1^Z = \frac{h''}{\bigl(1+h'^2\bigr)^{3/2}}.
\] 
\begin{enumerate}
    \item [(i)] Suppose that the directrix $M_1^{\,n-1}$ is not totally umbilical.\\
From Proposition \ref{Prop2} with $s=n-1$, we have that $H_{n-1}=0$, if and only if, there exists a constant $C$ such that
\begin{eqnarray}
 \frac{H^Y_{s-1}}{H^Y_s}=C,\quad\text{and}\quad \frac{N_1^Z}{\kappa^Z_1} =(n-1)(h-C) .
\end{eqnarray}
Substituting $N^Z_1$ and $\kappa^Z_1$ into the second equation above, we obtain that $h=h(v)$ satisfies  
\begin{eqnarray}
    (n-1) h''(h-C) = 1 + (h')^2. \label{eqhteo1}
\end{eqnarray}

Therefore, by Lemma \ref{lema1}, there exists a hypersurface of dimension $n-1$, given by $\widetilde{Y}=Y+CN^Y$, with $H^{\widetilde{Y}}_{n-2}=0$ and $N^Y=N^{\widetilde{Y}}$. Hence, the parameterization $X$ of $M^n$ can be rewritten as
\begin{eqnarray}
    X(u,v) \,=\, \widetilde{Y}(u) \,+\, (h(v)-C)\, N^{\widetilde{Y}}(u) \,+\, v\, e_{n+1},\label{XcomYtiu}
\end{eqnarray}
that is, the directrix of $M^n$ is given by $\widetilde{Y}$ and the profile is locally given by $\displaystyle{(h(v)-C,\,v)}$, where $h$ satisfies (\ref{eqhteo1}).

Using Lemma \ref{lemaedo} with constants $\widetilde{C}=0$, $\widetilde{a}=n-1$ and $\widetilde{b}=-(n-1)C$, we conclude that the profile of $M$ is locally parameterized by (\ref{Zteorema1}).

Conversely, suppose that the profile of $M^n$ is locally parameterized by (\ref{Zteorema1}) and that the directrix of $M^n$ is locally parameterized by $\widetilde{Y}$ with $H^{\widetilde{Y}}_{n-2}=0$. \\
Since the curvature of the profile and $N_1^Z$ are given by 
\[
\kappa = \frac{-\cos^n\theta}{C_1(n-1)},\quad\text{and}\quad N^Z_1 = \langle N^Z, e_1 \rangle = -\cos\theta,
\]
where $N^Z$ denotes the unit normal of the profile of $M$, then, by Proposition \ref{Prop2}, we obtain
\[
n N^Z_1 \frac{H_{n-1}}{H_n} 
= -(n-1)C_1 \sec^{\,n-1}\theta 
+ \frac{C_1(n-1)\cos\theta}{\cos^n\theta} 
= 0.
\]
That is, $H_{n-1}=0$.

\item[(ii)] Suppose that the directrix $M_1^{\,n-1}=S^{\,n-1}(r)$, i.e., the $(n-1)$-sphere of radius $r$, with Gauss map $N^Y=rY$.\\
Note that in this case, $X$ is given by
\[X(u,v) \,=\,\frac{r+h(v)}{r} Y(u) \,+\, v\, e_{n+1}.\]
From Proposition \ref{Prop2} with $s=n-1$, we have that $H_{n-1}=0$, if and only if, $h$ satisfies
\begin{eqnarray}
    (n-1)h''(h+r) = 1 + (h')^2. \label{eqhteo1esfera}
\end{eqnarray}
Using Lemma \ref{lemaedo} with constants $\widetilde{C}=0$, $\widetilde{a}=n-1$ and $\widetilde{b}=(n-1)r$, we conclude that the profile of $M^n$ is locally parameterized by (\ref{Zteorema1}).

Conversely, suppose that the profile of $M^n$ is locally parameterized by (\ref{Zteorema1}) and that the directrix of $M^n$ is $\mathbb{S}^{n-1}(r)$. \\
Therefore, proceeding as in the previous case, from Proposition \ref{Prop2}, it follows that $H_{n-1}=0$.
\end{enumerate}

\end{proof}

\begin{remark}
    The rotated translational hypersurfaces $M^n$ with $H_{n-1}=0$ given by Theorem \ref{teorema1} are locally parameterized on an open subset of $\mathbb{R}^n$ where 
\[
1 - C_1 \kappa_i \sec^{\,n-1}\theta \neq 0, \qquad C_1>0,\quad 1 \leq i \leq n-1,
\]
with $\kappa_i$ being the principal curvatures of the directrix of $M$.
\end{remark}

As an immediate consequence of Theorem \ref{teorema1}, we obtain a local classification of all rotated translational hypersurfaces of dimension $n=3$ whose second mean curvature vanishes, that is, $H_2=0$. More precisely, these form a two-parameter family.

\vspace{0.2cm}

\begin{theorem}\label{teorema3}
Let $M^3$ be a rotated translational hypersurface whose profile is a curve. 
Then the second mean curvature of $M^3$ is zero, that is, $H_{2}=0$, if and only if there exist constants $C_1>0$ and $C_2\in\mathbb{R}$ such that $M^3$ can be parameterized by one of the following parameterizations
\begin{enumerate}
    \item [(i)] $X:U_1\subseteq\mathbb{R}^3\to M^3$, where $U_1=\{(u,v,\theta)\,|\,\,1+C_1^{2}K\sec^{4}\theta\neq0\}$,
    \begin{eqnarray}
    X(u,v,\theta) = Y(u,v) + C_1\sec^2\theta N^Y(u,v) + \big(C_2+2C_1\tan\theta\big)e_4,\label{Xminimo}
\end{eqnarray}
where $Y$ is a minimal surface whose Gauss map is $N^Y$ and $K$ denotes the Gaussian curvature of $Y$.
\item [(ii)]  $X:U_2\subseteq\mathbb{R}^3\to M^3$, where $U_2=\{(u,v,\theta)\,|\,\,C_1\sec^2\theta\neq0\}$,
\begin{eqnarray}
    X(u,v,\theta) = \frac{C_1\sec^2\theta}{r}Y(u,v) + \big(C_2+2C_1\tan\theta\big)e_4,\label{Xesfera}
\end{eqnarray}
where $Y$ is the sphere of radius $r$, i.e. $S^{2}(r)$, with Gauss map $N^Y=rY$.
\end{enumerate}
\end{theorem}

\begin{remark}
    If the directrix hypersurface is the $s$-dimensional sphere $S^{s}(r)$ of radius $r$, then the behavior of the rotated translational hypersurfaces depends on the dimension of the profile.
    \begin{itemize}
    \item When the profile is a curve (that is, $t=1$), the rotated translational hypersurfaces is precisely a hypersurface of revolution.    
    \item When the profile has dimension $t \ge 2$, the immersion remains invariant
under the rotation group $O(s{+}1)$ acting on the spherical directions. 
Geometrically, $M$ is a warped product of the form
\[
    M = S^{s} \times_{\rho} M^{t}_{2},
\]
with warping function
\[
    \rho(u,v) = \frac{r + Z_{1}(v)}{r}.
\]
\end{itemize}

\end{remark}

\begin{remark}
    Let $M^3$ be a rotated translational hypersurface whose profile is a surface $M_2^2$. 
Then the second mean curvature of $M^3$ vanishes, i.e., $H_2=0$, if and only if there exists a constant $C>0$ such that the directrix of $M^3$ is a circle of radius $C$, and the radial mean curvature of the profile $M_2^2$ satisfies
\begin{eqnarray}
   N_1^Z A_Z = Z_1 - C,\label{perfilsuperficie} 
\end{eqnarray}

where $N^Z$ denotes the Gauss map of $M_2^2$.\\

Moreover, if $M_2^2$ is either a CPD surface, a translational surface, or a harmonic surface of graphic type, then $M_2^2$ is given by Theorems \ref{TeoremaCPD}, \ref{Teorematran}, and \ref{Teoremaharmonico}, for the constants $a = 1$ and $b = -C$.\\

If the rotated translational hypersurface $M^n$, with dimension $n>3$, 
has the profile surface $M_2^2$, then $H_{n-1}=0$ if and only if either there exists 
a hypersurface parallel to the directrix of $M^n$ whose $(n-3)$-mean curvature 
vanishes, that is, $\widetilde{H}_{n-3}=0$, or the directrix of $M^n$ is an 
$(n-2)$-dimensional sphere of radius $r$, and the radial mean curvature of the 
profile satisfies (\ref{perfilsuperficie}).

\end{remark}

\vspace{0.3cm}

\begin{remark}
    Denote by $X^{(3)}(u,v,\theta_1)$ the parameterization of the $3$-dimensional hypersurface given by Theorem \ref{teorema3}, and by $N^{(3)}$ its Gauss map. Then, by Theorem \ref{teorema1}, we obtain a rotated translational hypersurface $M^{4}$, parameterized by $X^{(4)}(u,v,\theta_1,\theta_2)$ with $H_3=0$, whose directrix is $X^{(3)}(u,v,\theta_1)$ and whose profile is a curve, given by
\[X^{(4)}(u,v,\theta_1,\theta_2)=X^{(3)}(u,v,\theta_1)+C_1\sec^3\theta_2\,N^{(3)}+\frac{3C_1}{2}\big(\sec\theta_2\tan\theta_2+\ln(\sec\theta_2+\tan\theta_2)\big)e_{5}.\]
Thus, using $X^{(4)}(u,v,\theta_1,\theta_2)$ as given above, and assuming that $N^{(4)}$ is its Gauss map, we can apply Theorem \ref{teorema1} once again to obtain a $5$-dimensional rotated translational hypersurface $X^{(5)}(u,v,\theta_1,\theta_2,\theta_3)$ with $H_4 = 0$.

Precisely, the Theorem \ref{teorema1} provides us with a recursive process to construct $m$-dimensional rotated translational hypersurface $X^{(m)}$, such that $H_{m-1} = 0$
\[X^{(m)}=X^{(m-1)}+C_1\sec^{m-1}\theta_{m-2}N^{(m-1)}+(m-1)C_1\int\sec^{m-1}\theta_{m-2}d\theta_{m-2}\,e_{m+1},\]
where $N^{(m-1)}$ denotes the Gauss map of $X^{(m-1)}$.

\end{remark}

\vspace{0,3cm}

Next, we provide the proof of Theorem \ref{teorema2}, which also gives a recursive procedure
to construct rotated translational hypersurfaces with the prescription of the
radial mean curvature in terms of the height and angle functions.\\

\begin{proof}[Proof of Theorem \ref{teorema2}]
 Let $\displaystyle{X(u,v) \,=\, Y(u) \,+\, Z_1(v) N^Y(u)\, + \sum_{r=2}^{n-s+1}Z_r(v) e_{s+r}}$, thus, by Proposition \ref{Prop2},
\[
N^{X}_{1}A_{X} \;=\; aX_{1}+b
\]
if and only if
\begin{eqnarray}
    aX_1+b=N_1^X\left(\frac{A_Y}{N_1^Z}+A_Z-\frac{sZ_1}{N_1^Z}\right).\label{eqX1Y1Z1}
\end{eqnarray}
Since $X_{1}=Y_{1}+Z_{1}N_{1}^{Y}$ and $N_{1}^{X}=N_{1}^{Y}N_{1}^{Z}$, then (\ref{eqX1Y1Z1}) is equivalent to
\[-N_1^YA_Y+aY_1+b=N_1^Y\big(-sZ_1+N_1^ZA_Z-aZ_1\big).\]
Noting that $-N_{1}^{Y}A_{Y}+aY_{1}+b$ and $N_{1}^{Y}$ depend on $u=(u_{1},\ldots,u_{s})$, 
while $-sZ_{1}+N_{1}^{Z}A_{Z}-aZ_{1}$ depends on $v=(u_{s+1},\ldots,u_{n})$, 
there exists a constant $C \in \mathbb{R}$ such that
\begin{eqnarray}
     N^{Y}_{1}(A_Y+C) = aY_{1}+b, \quad\text{and}\quad N^{Z}_{1}A_{Z} = (a+s)Z_{1}+C.\nonumber
 \end{eqnarray}
Therefore, by Lemma \ref{lema2} defining $\widetilde{Y} \;=\; Y - \frac{C}{s+a}\,N^{Y}$, the proof is complete.
\end{proof}

\begin{remark}
    By Theorem \ref{teorema2}, in order to construct an $n$-dimensional rotated translational hypersurface $M^n$ parametrized by (\ref{eqX}) and with directrix and profile given by (\ref{eqY}) and (\ref{eqZ}), respectively
  such that 
\[
N_{1}^{X}A_{X} \;=\; aX_{1}+b,
\] 
it suffices to know two hypersurfaces $\widetilde{Y}$ and $Z$ of dimensions $s$ and $n-s$, respectively, such that 
\[
N_{1}^{\widetilde{Y}}A_{\widetilde{Y}} \;=\; a\widetilde{Y}_{1}+b
\quad\text{and}\quad
N_{1}^{Z}A_{Z} \;=\; (a+s)Z_{1}+c.
\]
To determine the hypersurfaces $\widetilde{Y}$ and $Z$, it is enough to apply Theorem \ref{teorema2} to each of them.  
We can keep applying Theorem \ref{teorema2} until $\widetilde{Y}$ and $Z$ are both plane curves.

\end{remark}

\begin{example}
In this example, we construct a $3$-dimensional hypersurface $X$ satisfying 
\[
N_{1}^{X}A_{X} \;=\; aX_{1}+b.
\] 
Suppose that 
\[
X(u_{1},u_{2},u_{3}) \;=\; \widetilde{Y}(u_{1},u_{2}) \,+\, \left(h(u_{3})+\tfrac{C_1}{2+a}\right)N^{\widetilde{Y}} \,+\, u_{3}e_{4}.
\] 
By Theorem \ref{teorema2}, we have 
\[
N_{1}^{\widetilde{Y}}A_{\widetilde{Y}} \;=\; a\widetilde{Y}_{1}+b,
\]
and $h$ satisfies
\begin{eqnarray}
    1+h'^{2} \;=\; h''\bigl((2+a)h+C_1\bigr).\label{eqXexh}
\end{eqnarray}

Now suppose that 
\[
\widetilde{Y}(u_{1},u_{2}) \;=\; \alpha(u_{1}) \,+\, \left(f(u_{2})+\tfrac{C_{2}}{1+a}\right)N^{\alpha} \,+\, u_{2}e_{3},
\] 
where $\alpha(u_{1})$ is a curve which, for simplicity, we assume to be of graph type, that is, 
\[
\alpha(u_{1}) = \bigl(g(u_{1}),\,u_{1}\bigr).
\] 
In this case, 
\[
N^{\alpha} \;=\; \frac{1}{\sqrt{1+g'^{2}}}\,\bigl(1,\,-g'(u_{1})\bigr).
\]
Using Theorem \ref{teorema2} again, we have that 
\[
N_{1}^{\widetilde{Y}}A_{\widetilde{Y}} \;=\; a\widetilde{Y}_{1}+b,
\] 
if and only if $g$ and $f$ satisfy, respectively, 
\begin{eqnarray}
    1+g'^{2} \;=\; g''(ag+b) 
\quad \text{and} \quad 
1+f'^{2} \;=\; f''\bigl((a+1)f+C_{2}\bigr),\label{eqXexfeg}
\end{eqnarray}
since in this case 
\[
A_{\alpha}=\frac{1}{\kappa^{\alpha}}
\quad \text{and} \quad 
A_{Z}=\frac{1}{\kappa^{Z}},
\]
where $Z=(f(u_{2}),u_{2})$.\\

Therefore, we have constructed a $3$-dimensional hypersurface $X$ that satisfies
\[
N_{1}^{X}A_{X} \;=\; aX_{1}+b,
\]
given by
\[
X(u_{1},u_{2},u_{3})
 \;=\; \alpha(u_{1})
 \,+\, \Bigl(f(u_{2})+\tfrac{C_{2}}{1+a}\Bigr)N^{\alpha}
 \,+\, u_{2}e_{3}
 \,+\, \Bigl(h(u_{3})+\tfrac{C_1}{2+a}\Bigr)N^{\widetilde{Y}}
 \,+\, u_{3}e_{4},
\]
where $\alpha(u_{1})=(g(u_{1}),u_{1})$, 
\[
N^{\alpha}=\frac{1}{\sqrt{1+g'^{2}}}\,(1,-g'),\quad\text{and}\quad N^{\widetilde{Y}}=\frac{-N^{\alpha}+f'e_3}{\sqrt{1+f'^{2}}}
\]
where the functions $g, f, h$, given by Lemma \ref{lemaedo}, which depend on the constants $(\widetilde{C}, \widetilde{a}_1, \widetilde{b}_1)$, are obtained by using, respectively, $(0,a,b)$, $(0,a+1,C_2)$, and $(0,a+2,C_1)$.
\end{example}

\vspace{0.2cm}

\bigskip

\end{document}